\documentclass[11pt]{amsart}
\usepackage{amsmath}
\usepackage{amsthm}
\usepackage[T1]{fontenc}
\usepackage[sc]{mathpazo}
\usepackage{enumerate}

\usepackage{etoolbox}
\patchcmd{\section}{\scshape}{\bfseries}{}{}

\newenvironment{nouppercase}{%
  \renewcommand{\uppercasenonmath}[1]{}}{}

\newcommand{\C}{{\mathbb{C}}}
\newcommand{\R}{{\mathbb{R}}}

\newcommand{\Z}{{\mathbb{Z}}}

\newcommand{\intersect}{\cap}

\newcommand{\conj}[1]{\overline{#1}}

\DeclareMathOperator{\Log}{Log}

\newtheorem{theorem}{Theorem}[section]
\newtheorem{proposition}[theorem]{Proposition}
\newtheorem{lemma}[theorem]{Lemma}

\newtheorem{corollary}[theorem]{Corollary}
\theoremstyle{definition}

\theoremstyle{remark}

\newtheorem*{remark*}{Remark}
\theoremstyle{example}

\newcommand{\PosOrthant}{\R_+^n}

\begin{document}
\title[~]
{Characterization of polynomials whose large powers \\
have all positive coefficients}
\author[~]
{Colin Tan and Wing-Keung To}

\address{Colin Tan, Department of Statistics \& Applied Probability, National University of Singapore, 
Block S16, 6 Science Drive 2, Singapore 117546
}
\email{statwc@nus.edu.sg}
\address{Wing-Keung To,
Department of Mathematics, National University of Singapore, Block
S17, 10 Lower Kent Ridge Road, Singapore 119076}
\email{mattowk@nus.edu.sg}

\thanks{Wing-Keung To was partially supported by the research grant
R-146-000-142-112 from the National University of Singapore and
the Ministry of Education.
}

\begin{nouppercase}
\maketitle
\end{nouppercase}
\numberwithin{equation}{section}

\begin{abstract}
We give a criterion which characterizes a 
homogeneous real multi-variate polynomial to have the property that
all sufficiently large powers of the polynomial (as well as their products with any given positive 
homogeneous polynomial) have all positive coefficients.  Our result generalizes a result of De Angelis, which 
corresponds to 
the case of homogeneous bi-variate polynomials, as well as a classical result of P\'olya, which corresponds to 
the case of a specific linear polynomial.  As an application, we also give a characterization of certain 
polynomial spectral radius functions of the defining matrix functions of Markov chains.
\end{abstract}

\section{Introduction and main results} \label{sec: intro}

Positivity conditions for polynomials with real coefficients are relevant in several branches of pure and 
applied mathematics, including real algebraic geometry, convex geometry, probability theory and 
optimization, and have been extensively studied (see e.g. \cite{CD97, Deangelis94, 
Deangelis03, Han86, MT93, Polya28, Reznick95, TY06}  and the references therein). An important class of polynomials are those whose coefficients are positive. 

De Angelis \cite{Deangelis03} characterized those univariate polynomials $p$ such that 
	$p^m$ has all positive coefficients for all sufficiently large $m$
								in terms of certain positivity conditions on $p$ itself.
As an application, 
    he obtained a characterization of certain univariate polynomials $p$ 
        for which there exists an irreducible (or aperiodic) Markov chain whose defining matrix has $p$ as its 
        spectral radius function \cite[Theorem 6.7]{Deangelis942}. This spectral radius function
         is an important invariant in the study of Markov shifts
                (see e.g. \cite{MT93}).  As such, it is interesting and natural to ask whether similar results hold in the
                multivariate setting.
																			

In this paper, we generalize both the afore-mentioned results of De Angelis 
to the case of homogeneous multivariate polynomials.  Let $n \ge 1$.
A homogeneous polynomial $f= \sum_{|I| = d} c_{I} x^I \in \R[x_1\dots,x_n]$ 
	with real coefficients and of degree $d$ is said to 
		{\emph{have all positive coefficients}} if $c_I > 0$ for all $|I| = d$.
Here $I = (I_1,\ldots, I_n)$ is a multi-index
			of length $|I| := I_1 + \cdots + I_n$ and
				$x^{I} = x_1^{I_1} x_2^{I_2} \cdots x_n^{I_n}$.
Next we let
$
\mathbb{R}_+^{n} := \{ x=(x_1,\ldots,x_n)\in \R^{n } \, \big| \, \, x_i\geq 0,~i=1,\ldots,n\}
$
	denote the closed positive orthant in the real Euclidean space $\R^{n}$ (for simplicity, we also write 
	$\mathbb{R}_+:= \mathbb{R}_+^{1}$).
The circle group $U(1) := \{e^{i\theta} \, \big| \, \theta\in \R\}$ acts via pointwise multiplication on the
	complex Euclidean space $\C^{n}$, given by
		$e^{i\theta} \cdot z := (e^{i\theta} z_1 ,\dots, e^{i\theta} z_n)$ 
					for $e^{i\theta} \in U(1)$ and $z = (z_1, \dots, z_n) \in \C^{n } $.
The $U(1)$-invariant subset of $\C^n$ generated by $\PosOrthant$ is given by
$
U(1) \cdot \PosOrthant := \{ e^{i\theta} \cdot x \, \big| \, e^{i\theta} \in U(1), ~x \in \PosOrthant\}$.
For $k = 1,\ldots , n$, we denote the $k$-th facet of $\PosOrthant$ by	
$
F_k(\PosOrthant) := \{ x=(x_1,\dots,x_n) \in \PosOrthant \, | \,  x_k = 0\}$. 

\medskip
Our main result in this paper is the following:

\begin{theorem} \label{thm: mainThm}
Let $p \in \R[x_1,\dots,x_n]$ be a nonconstant homogeneous polynomial.
The following two statements are equivalent:
\begin{enumerate}[(a)]
\item \label{thm: Condition}
$p$ satisfies the following three conditions:
\begin{description}
\item[(Pos1)] \label{Pos1} $p(1, 0, \ldots, 0), p(0,1, 0, \ldots, 0), \dots , p(0, \ldots, 0, 1) > 0$.
\item[(Pos2)] \label{Pos2} For all $k  =1, \ldots, n$,
								 $\displaystyle\frac{\partial p}{ \partial x_k} (x) > 0$ 
  											for all  $ x \in F_k(\PosOrthant) \setminus \{0\}$, 
\item[(Pos3)]  \label{Pos3} $|p(z)| < p(|z_1|, \ldots, |z_n|)$ for all 
$z=(z_1,\dots,z_n) \in \C^n \setminus (U(1) \cdot \PosOrthant)$.
\end{description}

\item \label{thm: EventualPos} For each homogeneous polynomial $q \in\R[x_1,\dots,x_n]$
				such that $q(x) > 0$ whenever $x \in \mathbb{R}_+^{n} \setminus \{0\}$,
					there exists $m_o > 0$ such that for each integer $m \ge  m_o$, 
								\, $p^m \cdot q$ has all positive coefficients.
\end{enumerate}
\end{theorem}
The implication $\eqref{thm: Condition} \implies \eqref{thm: EventualPos}$ may be regarded as a 
Positivstellensatz for those homogeneous polynomials $q$ which are 
strictly positive on $\mathbb{R}_+^{n} \intersect \{p = 1\}$, as the latter condition is certified by 
	the algebraic property that $p^m \cdot q $ has all positive coefficients for some $m \ge 1$.

The bulk of our proof of the implication $\eqref{thm: Condition} \implies \eqref{thm: EventualPos}$ 
consists of showing that a certain  Hermitian bihomogeneous polynomial $P$ associated to $p$ satisfies the sufficient conditions of a Hermitian Positivstellensatz of Catlin-D'Angelo 
\cite{CD99} (see also Theorem \ref{thm: CDThm} below), which enables us to apply the latter result.  As mentioned above, Theorem 
\ref{thm: mainThm} for the case of $n = 2$ dehomogenizes to De Angelis' 
Positivstellensatz \cite[Theorem 6.6]{Deangelis03}. Thus our proof settles affirmatively 
a debate in MathOverflow \cite{Ere14} on whether De Angelis' Positivstellensatz is a consequence of 
Catlin-D'Angelo's Positivstellensatz. 

Several examples of homogeneous real polynomials $p$ which satisfy the three positivity conditions in 
$\eqref{thm: Condition}$ can be found in the literature. 
A classical example is the linear form $p = x_1 + \cdots + x_n$; in this case, 
Theorem \ref{thm: mainThm} is 
the Positivstellensatz of P\'olya on the simplex (\cite{Polya28}).  A more general example is given by any homogeneous polynomial $p$ which has all positive coefficients.
A different kind of example is the polynomial 
\begin{align}\label{eq: plambda}
p_{\lambda}(x_1, x_2) := (x_1+x_2)^{2k} - \lambda x_1^kx_2^k\quad\text{with  }
\binom{2k}{k}<\lambda < 2^{2k - 1}\text{ and }k \ge 2
\end{align}
given by D'Angelo-Varolin in \cite[Theorem 3]{DV04},
	for which the coefficient of $x_1^kx_2^k$ in $p_\lambda$
 is negative (we will skip the verification that $p_\lambda$ satisfies the three 
positivity conditions in $\eqref{thm: Condition}$, which is similar to the calculations given in 
\cite{DV04}).

The three positivity conditions in $\eqref{thm: Condition}$ are independent, in the sense that, any 
two of these conditions do not imply the third one.  Consider the following polynomials (with $n\geq 2$):
\begin{align}\label{eq: example1}
& (x_1 + \cdots + x_n)^3 - x_1^3,\\
\label{eq: example2}
& x_1^2(x_1 + \cdots + x_n) + (x_2 + \cdots + x_n)^3,\\
\label{eq: example3}
&(x_1 + \cdots + x_n)^4 - 8 x_1^2 x_2^2.
\end{align}
As the reader can verify easily, \eqref{eq: example1} satisfies (Pos2) and (Pos3) but violates (Pos1) at the point 
$(1, 0, \ldots, 0)$, \eqref{eq: example2}  satisfies (Pos1) 
and (Pos3) but violates (Pos2) on the facet $F_1(\PosOrthant)$, and \eqref{eq: example3} 
satisfies (Pos1) and (Pos2) but violates (Pos3) at the point $(-1, 1, 0, 0, \ldots, 0)$.

As pointed out by the referee, 
	in the special case when $p = x_1 + \cdots + x_n$
		and $q$ is as in Theorem \ref{thm: mainThm}\eqref{thm: EventualPos},
	the result of Halfpap-Lebl in \cite{HL13} yields a lower bound for $m_o$ 
		 in terms of the signature of $q$.	 		

In view of Theorem \ref{thm: mainThm}, it is natural and interesting to ask for a similar characterization of homogeneous 
polynomials whose large powers have all nonnegative coefficients.  It appears to the authors 
that the method in this paper does not generalize readily to handle such borderline case, and new ideas are 
needed
to tackle the problem.
To illustrate the subtlety of 
this problem, we mention that the polynomial in \eqref{eq: example3} (which,
in the case when $n=2$, corresponds to a
limiting case of the family of polynomials in \eqref{eq: plambda} with $k=2$ and $\lambda=8$) satisfies 
a weaker version of (Pos3) (with \lq$<$\rq~there replaced by  \lq$\leq$\rq), but it is easy to check that 
none of its powers has all nonnegative coefficients.



\medskip

Let $\Z_+:=\{k\in\Z\,\big|\, k\geq 0\}$, and denote by $\Z_+[x_1,\ldots, x_n]$ the 
semiring of polynomials in $x_1,\ldots, x_n$ with coefficients in $\Z_+$.   
Let $A$ be an {\it irreducible } (resp. {\it aperiodic}) 
square matrix 
over $\Z_+[x_1,\ldots, x_n]$
(see e.g.\ Section \ref{sec: Application} for the definitions).
Denote the spectral radius function of 
$A$ by $\beta_A=\beta_A(x_1,\ldots, x_n)$ .  

As an application of Theorem \ref{thm: mainThm}, we have

\begin{corollary} \label{cor: betaApplication}
Let $p \in \Z[x_1,\dots,x_n]$ be a homogeneous polynomial which satisfies (Pos1) and (Pos2).
The following statements are equivalent:
\begin{enumerate}[(i)]
\item \label{cor: pos3}$p$ satisfies (Pos3).
\item \label{cor: irr} $p = \beta_A$ for some irreducible square matrix $A$ over $\Z_+[x_1,\dots,x_n]$.
\item \label{cor: aper} $p = \beta_A$ for some aperiodic square matrix $A$ over 
$\Z_+[x_1,\dots,x_n]$.
\end{enumerate}
\end{corollary}

As mentioned earlier, the case of $n =2$ in Corollary \ref{cor: betaApplication} dehomogenizes to 
De Angelis' result \cite[Theorem 6.7]{Deangelis942}.
We refer the reader to Section \ref{sec: Application} for the interpretation of Corollary \ref{cor: betaApplication}
	in terms of Markov chains.

De Angelis' Positivstellensatz \cite[Theorem 6.6]{Deangelis03} has been applied by Bergweiler-Eremenko \cite{BE15}
	to study the distribution of zeros of polynomials with positive coefficients 
		(see also \cite{EF}).
An effective version of P\'{o}lya's Positivstellensatz by Powers-Reznick \cite{PR09} was applied by 
Schweighofer \cite{Schweighofer04} to obtain complexity bounds on a Positivstellensatz of 
Schm\"{u}dgen \cite{Schmudgen91},
		and by de Klerk-Pasechnik \cite{DP02} to
				estimate the rate of convergence of a certain hierarchy of conic linear programs to the 
				stability number of a graph. 			
As a generalization of the Positivstellensatze of De Angelis and P\'olya,
Theorem \ref{thm: mainThm}
  may also have similar applications,
  	which will not pursued here.

The organization of this paper is as follows. In Section \ref{sec: CDThm}, we recall 
some background material
on
bihomogeneous polynomials. 
In Section \ref{sec: positivity}, we relate some positivity properties of a homogeneous real
polynomial with those of its associated bihomogeneous polynomial.
In Section \ref{sec: proofOfTheorem}, we give the proof of Theorem \ref{thm: mainThm}.  
In Section \ref{sec: Application}, we give the deduction of Corollary \ref{cor: betaApplication}.
 
 \medskip
\textbf{Acknowledgements.}
The first author would like to thank Ser Peow Tan, Yan Loi Wong and Xingwang Xu 
	for suggesting to work in the direction of this problem
		and is grateful for his wife's encouragement to complete this article.
The authors also acknowledge 
	John P. D'Angelo,
	Valerio De Angelis,
	Alexandre Eremenko, 
	David Handelman,
	and John Jiang
	    for sharing their work and for helpful discussions.
	   The authors are grateful to the referee for numerous comments and suggestions leading to the paper in its present form.

\section{Bihomogeneous polynomials and Catlin-D'Angelo's Positivstellensatz}
	\label{sec: CDThm}

In this section, we recall some background material regarding bihomogeneous polynomials, 
which is mostly taken from \cite{CD97, CD99, DAngelo02, Var08}.  
Throughout this section, we fix a positive integer $n\geq 2$.    
Denote by $\C[z_1, \ldots, z_n, \conj{w_1}, \ldots, \conj{w_n}]$
	the complex polynomial algebra in the indeterminates $z_1, \ldots, z_n, \conj{w_1}, \ldots, \conj{w_n}$.  
For $d\geq 0$, a polynomial $P\in \C[z_1, \ldots, z_n, \conj{w_1}, \ldots, \conj{w_n}]$ is said to be {\emph{bihomogeneous of 
bidegree $(d, d)$} if 
\begin{equation}\label{Pzw}
P(\zeta z,\conj{\mu} \conj{w})=\zeta^d\conj{\mu}^d  P(z,\conj{w})
\end{equation}
for all $\zeta,\mu\in \C$ and $z=(z_1, \ldots, z_n),~w=(w_1, \ldots, w_n)\in \C^n$.
Such $P$ is said
 to be {\emph{Hermitian}}  if 
$
\conj{P(z, \conj{w})} = P(w, \conj{z})$ for all $z, w\in \C^{n}$.   Furthermore,
$P$ is said to be {\it positive on }$\C^{n}\setminus\{0\}$ if 
$P(z, \conj{z}) > 0$ for all $ z \in \C^{n}\setminus\{0\}$.

\medskip
 For $d\geq 0$, we denote by $ \C[z_1, \ldots, z_n]_d$ the complex vector space  
		of homogeneous holomorphic polynomials in $\C^n$ of degree $d$. 
		A Hermitian bihomogeneous polynomial $P$ is said to be a {\emph{maximal squared norm}} 
if there exists a basis 
$\{g_1,\dots, g_N\}$ of $ \C[z_1, \ldots, z_n]_d$ (with $N= \dim_{\C} \C[z_1, \ldots, z_n]_d = \binom{d + n - 1}{n - 1}$) such that 
\begin{equation}\label{eq: sosP}
P(z, \conj{w}) = \sum_{k = 1}^N g_k(z) \cdot \conj{g_k(w)}\quad\text{for all }z,w\in\C^n
\end{equation}
(so that $ P(z, \conj{z})= \sum_{k = 1}^N |g_k(z)|^2$ for all $z\in\C^n$).    

\medskip	
From \eqref{Pzw}, one easily sees that a
Hermitian bihomogeneous polynomial $P\in \C[z_1, \ldots, z_n, \conj{w_1}, \ldots, \conj{w_n}]$ of bidegree
$(d,d)$ may be regarded as a Hermitian form on the dual vector space of $ \C[z_1, \ldots, z_n]_d$.  In
particular, with respect to any basis $\{h_1,\dots, h_N\}$ of $ \C[z_1, \ldots, z_n]_d$, there exists a unique $N\times N$ 
Hermitian 
matrix $C=\big(
c_{k\conj{l}})_{1\leq k,l\leq N}$ such that 
\begin{equation}\label{eq: matrixC}
P(z,\conj{w})=\sum_{1\leq k,l\leq N}c_{k\conj{l}}h_k(z)\conj{h_l(w)}
\end{equation}
for all $z,w\in\C^n$.  It is easy to see that $P$ is a maximal squared norm
if and only if its associated matrix $C=\big(
c_{k\conj{l}})$ with respect to some (and hence any) basis of $ \C[z_1, \ldots, z_n]_d$ is positive definite.  
Note that a 
Hermitian bihomogeneous polyomial positive on $\C^n\setminus\{0\}$ 
	need not be a maximal squared norm.

\medskip
Following \cite{CD99}, 
a Hermitian bihomogeneous polynomial $P$ is said to satisfy
	{\emph{the strong global Cauchy-Schwarz (in short, SGCS) inequality}}
	if 
	\begin{equation}\label{eq: SGCS}|P(z, \conj{w})|^2 < P(z, \conj{z}) P(w, \conj{w})
	 	\quad\text{for all linearly independent }z, w\in \C^{n},
	 	\end{equation}
	 	i.e., the above inequality holds whenever $z$ and $w$ are not scalar multiples of each other.
(Note that the Hermitian bihomogeneity of $P$ implies that 
$|P(z, \conj{w})|^2 = P(z, \conj{z}) P(w, \conj{w})$ whenever $z$ and $w$ are linearly dependent.)  
We recall the following result of Catlin-D'Angelo:

\begin{theorem}[{\cite[Theorem 1, Corollary and its proof]{CD99}}]\label{thm: CDThm}
Let $P \in \C[z_1, \ldots, z_n, \conj{w_1}, \ldots, \conj{w_n}]$ be a non-constant Hermitian bihomogeneous polynomial such that 
(i) $P$ is positive on $\C^n\setminus\{0\}$, (ii) the domain $\{z \in \C^{n} : P(z, \conj{z}) < 1\}$
		is strongly pseudoconvex, and (iii) $P$ satisfies the SGCS inequality.  
		Then for each Hermitian bihomogeneous polynomial $Q \in \C[z_1, \ldots, z_n, \conj{w_1}, \ldots, \conj{w_n}]$
			positive on $\C^n \setminus\{0\}$,
        there exists $m_o > 0$ such that for each integer $m \ge m_o$,
            $P^m \cdot Q$ is a maximal squared norm.
\end{theorem}

\section{Homogeneous polynomials and associated bihomogeneous polynomials} 
\label{sec: positivity}

Throughout this section, we fix a positive integer $n\geq 2$.  
For each homogeneous real polynomial $p =\sum_{|I| = d} c_{I} x^I \in \R[x_1,\dots,x_n] $ of degree 
$d$, we have an associated Hermitian bihomogeneous polynomial $P$ of bidegree $(d,d)$
            given by
\begin{equation} \label{eq: defAssocBihomPoly}
	P(z, \conj{w}) := p(z_1 \conj{w_1}, \dots, z_n \conj{w_n})=\sum_{|I| = d} c_{I} z^{I} \conj{w^{I}}
\end{equation} 
for all $z=(z_1,\dots,z_n), ~w=(w_1,\dots, w_n)\in\C^n$.
We remark that the bihomogeneous polynomial $P$ is indeed Hermitian, since 
 the $c_I$'s are real.  First we make a simple observation as follows:
 
\begin{proposition}\label{prop: AssocBihomIsSOS}
$p$ has all positive coefficients
    if and only if
      $P$
            is a maximal squared norm.
\end{proposition}

\begin{proof}  With notation as in Section \ref{sec: CDThm},
 the monomials $\{z^I\}_{|I| = d}$ form a basis of the 
    complex vector space $\C[z_1,\dots, z_n]_{d}$.  With respect to this basis, it follows readily from 
    \eqref{eq: defAssocBihomPoly} that the square matrix associated to $P$ (as in \eqref{eq: matrixC}) is given by 
    the real diagonal matrix $C:=\mathop{\mathrm{diag}}(c_I)_{|I|=d}$.  Then, as remarked in 
    Section \ref{sec: CDThm},
  $P$ is a maximal squared norm  
    if and only if the matrix $C$ is positive definite.  
    In turn, the latter condition holds if and only if $c_I>0$ for all $|I|=d$.
\end{proof}

 Our main result in this section is the following proposition:
\begin{proposition} \label{prop: P3properties}
Let $p\in\R[x_1,\dots,x_n]$ be a nonconstant homogeneous real polynomial, and let $P$ be its associated Hermitian bihomogeneous polynomial 
as in \eqref{eq: defAssocBihomPoly}.  
If $p$
    satisfies (Pos1), (Pos2), and (Pos3), then the following statements hold:
    \begin{enumerate}[(i)]
\item \label{prop: Pposdef}      $P$ is positive on $\C^n\setminus\{0\}$.
\item \label{prop: spsc}  The domain $\Omega_{P < 1}:=\{z\in\C^n\,\big|\, P(z,\conj{z})<1\}$ is strongly pseudoconvex.
\item \label{prop: PSGCS} $P$ satisfies the SGCS inequality.
\end{enumerate}
\end{proposition}

Througout the rest of this section, which is devoted to the proof of the above proposition, we let 
$p\in\R[x_1,\dots,x_n]$ be a nonconstant homogeneous real polynomial, and let $P$ be its 
associated Hermitian bihomogeneous polynomial.  

\begin{proposition}\label{prop: Pos1Pos3impliesPpd} 
If $p$ satisfies (Pos1) and (Pos3), 
    then 
\begin{enumerate}[(i)]
\item \label{prop: ppositive} $p(x)>0 $ for all $x\in\R_+^n \setminus \{0\}$, and 
\item \label{prop: Ppdf} $P$ is positive on $\C^n\setminus\{0\}$.
\end{enumerate}
\end{proposition}

\begin{proof}
Let $x=(x_1,\dots,x_n) \in \R_+^n \setminus \{0\}$ be given.  If $x_i>0$ for only one $i$ (with $1\leq i\leq n$), then it follows 
readily from (Pos1) and the homogeneity of $p$ that $p(x)>0$.  
If $x_i , x_j > 0$ for some $1 \le i < j \le n$, then by permuting the coordinate functions,
we may assume without loss of generality that $x_1, x_2>0$.  Then 
one easily checks that $x^\prime:=(-x_1,x_2,x_3,\dots,x_n)\in \C^n \setminus (U(1) \cdot \PosOrthant)$ 
(since the equalities $e^{i\theta}\cdot x_1=-x_1$ and $e^{i\theta}\cdot x_2=x_2$ imply 
$e^{i\theta}=-1$ and $e^{i\theta}=1$ respectively, which is a contradiction).  Then by (Pos3), one has 
$ |p(x^\prime)|<p(x)$, which implies that $p(x)>0$ again. 
This finishes the proof of \eqref{prop: ppositive}. 
For \eqref{prop: Ppdf}, we let $  z=(z_1,\dots,z_n)\in \C^n\setminus\{0\}$ be given.  Then one sees
from \eqref{eq: defAssocBihomPoly} and \eqref{prop: ppositive} that
$       
    P(z,\conj{z})  =    p(x)>0$, where  $x=(|z_1|^2,\dots,|z_n|^2)\in\R_+^n\setminus\{0\}$.
Hence $P$ is positive on $\C^n\setminus\{0\}$.
\end{proof}

Next we recall a result of De Angelis \cite{Deangelis94}.  For $\ell \geq 1$, we denote the interior of 
$\R^\ell_+$ by $(\R^\ell_+)^\circ:=\{(s_1,\dots, s_\ell)\in \R^\ell\,\big|\, s_i>0,~i=1,\dots,\ell\}$.  
Let $f(s) = \sum_{I} c_{I} s^{I} \in \R[s_1, \ldots, s_\ell]$ be a (possibly non-homogeneous) polynomial 
	such that $f(s) > 0$ for all $s =(s_1,\dots, z_\ell)\in (\R^\ell_+)^\circ$.  Consider the set
$\Log(f):=\{I\in\Z^\ell\,\big|\,\, c_{I} \neq 0 \}$, 
and recall that
	the {\emph{Newton polytope}} $N(f)$ of $f$ is defined as the convex hull
	of 
	$\text{Log}(f)$ in $\R^\ell$.
We associate to $f$ the $\ell \times \ell$ matrix-valued function $J_f
: (\R^\ell_+)^\circ\to \R^{\ell^2}$ whose components are given by
\begin{align} \label{eq: defineJ}
J_f(s)_{ij} :&= s_j \cdot \frac{\partial}{\partial s_j}\Big(s_i\cdot\frac{\partial }{\partial s_i}\big(\log f\big)\Big)(s)
\\
&=s_is_j\frac{\partial^2 } {\partial s_i\partial s_j  }\big(\log f\big)(s)+\delta_{ij}\cdot s_j\cdot
\frac{\partial}{\partial s_i}\big(\log f\big)(s)
\notag
\end{align}
for $s\in (\R^{\ell}_+)^\circ,~1\leq i,j\leq \ell$.  Here $\delta_{ij}$ denotes the Kronecker delta. i.e., $\delta_{ij}=1 $ (resp.
$0$) if $i=j$ (resp. $i\neq j$).  Next we introduce a change of variables, and consider 
the function $\widetilde f:\R^{\ell}\to \R$ associated to $f$ given by 
\begin{equation}\label{eq: tildef}
\widetilde f(t)=f(e^{t_1}, \ldots, e^{t_\ell}) \quad \text{for } t = (t_1,\ldots, t_\ell) \in \R^{\ell}.
\end{equation}
Using \eqref{eq: defineJ}, one easily checks that the Hessian matrix of 
$\log\widetilde f$ coincides with $J_f$, i.e., one has
\begin{equation}\label{eq: HessJlogf}
\frac{\partial^2 } {\partial t_i\partial t_j  }\big(\log \widetilde f\big)(t) =J_f(e^{t_1}, \ldots, e^{t_\ell})_{ij}
\end{equation}
for all $t=(t_1,\dots, t_\ell)\in\R^{\ell},~1\leq i,j\leq \ell$.   We recall the following result:
	
\begin{lemma}[De Angelis {\cite[Theorem 6.11]{Deangelis94}}] \label{lem: positiveHessianInInterior}
Let $f(s) \in \R[s_1, \ldots, s_\ell]$ be a polynomial such that $f(s) > 0$ for all $s \in (\R^{\ell}_+)^\circ$.
Suppose that there exists an open neighborhood $V$ of $(\R^{\ell}_+)^\circ$ in $(\C \setminus \{0\})^\ell$
		such that $|f(z)| \le f(|z_1|, \ldots, |z_\ell|)$ for all $z = (z_1, \ldots, z_\ell) \in V$, and the 
		Newton polytope $N(f)$ has affine dimension $\ell$.
Then the $\ell \times \ell$ matrix $J_f(s)$ is positive definite for all $s \in (\R^{\ell}_+)^\circ$.
\end{lemma}

As before, we let $p\in\R[x_1,\dots,x_n]$ be a nonconstant homogeneous polynomial. Let 
$\mathfrak{S}_n$ denote the group of
permutations of the coordinate functions on $\R^n$.  For each $1\leq \ell\leq n-1$ and each 
$\sigma\in \mathfrak{S}_n$, we 
associate to $p$ a non-homogeneous polynomial $p_{\ell,\sigma}\in \R[s_1,\dots, s_\ell]$ 
given by
\begin{equation}\label{eq: definepmsigma}        
p_{\ell,\sigma}(s_1,\dots, s_\ell):=p(\sigma(s_1,\cdots, s_\ell,0,\dots,0,1)).
\end{equation} 

\begin{lemma} \label{lem: Zspan}  
(i)  If $p$ satisfies (Pos1), then for each $1\leq \ell\leq n-1$ 
and each $\sigma\in
\mathfrak{S}_n$, the Newton polytope $N(p_{\ell,\sigma})$ has affine dimension $\ell$. 
\par\noindent
(ii)  If $p$ satisfies (Pos1) and (Pos2), then for each $1\leq \ell\leq n-1$ 
and each $\sigma\in
\mathfrak{S}_n$, the set
$
S_{p_{\ell,\sigma}} : =\{I - J \, \big| I,\,J\in \Log (p_{\ell,\sigma})\}
$
generates $\Z^{\ell}$ as a 
$\Z$-module.
\end{lemma}   

\begin{proof}  As the proofs of the lemma for all the $p_{\ell,\sigma}$'s are the same, 
 we will only prove the lemma for the case when $\sigma$ is the identity permutation, so that
$p_{\ell,\sigma}(s_1,\dots,s_\ell)=p(s_1,\cdots, s_\ell,0,\cdots,0,1)$.  Let $p$ be of degree $d\geq 1$.    
If $p$ satisfies (Pos1), then it follows readily that $\Log (p)$ contains the points
$(d,0,\dots,0)$, $\dots$, $(0,\dots,0,d)$.  Hence $
\Log (p_{\ell,\sigma})(\subset\Z^\ell)$ contains the points $(d,0,\dots,0)$, $\dots$, $(0,\dots,0,d)$
and $(0,\dots,0)$.  This implies readily that $N(p_{\ell,\sigma})$ has affine dimension $\ell$, and this finishes the proof of
(i).  We proceed to prove (ii).  
For each 
$1\leq i\leq \ell$, one 
easily checks that
\begin{equation}\label{eq: partialpmsigma}
\frac{\partial p_{\ell,\sigma}}{\partial s_i}(0,\dots,0)=\frac{\partial p}{\partial x_i}(0,\dots,0,1)>0,
\end{equation}
where the inequality holds since $p$ satisfies (Pos2) and $(0,\dots, 0,1)\in F_i(\PosOrthant)$. 
This implies that $\Log (p_{\ell,\sigma})(\subset \Z^{\ell})$ contains the points $(1,0,\dots,0)$, $\ldots$, $(0,\dots, 0,1)$, and so does $S_{p_{\ell,\sigma}}$ 
(since $\Log (p_{\ell,\sigma})$ also contains $(0,\cdots, 0)$ as shown in (i)). 
It follows that $S_{p_{\ell,\sigma}}$ generates $\Z^{\ell}$ as a 
$\Z$-module.
\end{proof} 

\begin{proposition} \label{prop: StronglyPseudoconvex}
If $p$ satisfies (Pos1), (Pos2) and (Pos3),
    then the domain $\Omega_{P < 1}$ is strongly pseudoconvex.
\end{proposition}

\begin{proof} Let $p$ be of degree $d\geq 1$.
From Proposition \ref{prop: Pos1Pos3impliesPpd}, one knows that $P(z,\conj{z})>0$ for all 
$z\in \C^n\setminus\{0\}$.  Together with the bihomogeneity of $P$ of bidegree $(d,d)$ with $d\geq 1$, it follows readily that $\Omega_{P < 1}$ is a bounded domain in 
$\C^n$ with smooth boundary.  Note that we may write 
$\Omega_{P < 1}=\{z\in\C^n\,\big|\, \log P(z,\conj{z})<0\}$.
To prove the proposition, it suffices to show
that 
\begin{align}\label{eq: pconvex}
(\sqrt{-1}\partial\conj{\partial}\log P)(v,\conj{v})&>0\quad\text{for any }z^*\in
\partial \Omega_{P < 1} \text{ and }\\
\text{any }&0\neq v\in T_{z^*}(\C^n)\text{ satisfying }\partial(\log P)(v)=0.
\notag
\end{align}
(Here $P$ denotes $P(z,\conj{z})$.) Regarding $\C^n$ as a complex manifold, it is well-known that one only needs to verify \eqref{eq: pconvex}
in terms of some local (possibly non-Euclidean) holomorphic coordinate system at each $z^*\in
\partial \Omega_{P < 1}$
(see e.g. \cite[p. 66]{FG2002}).  Take an arbitrary point $z^*=(z_1^*,\dots, z_n^*)\in \partial \Omega_{P < 1}$,
so that $P(z^*,\conj{z^*})=1$ (and thus $z^*\neq 0$).  
By
permuting the coordinate functions, we will assume without loss of generality that 
$z_n^*\neq 0$.
Next we introduce a new local coordinate system $u$ near $z^*$ 
via the holomorphic map $\phi:\{u=(u_1,\dots,u_n)\in\C^n\,\big|\, u_n\neq 0\}\to\C^n$ given by
\begin{equation}\label{eq: coordu}
z=\phi(u):=(u_1u_n,\dots, u_{n-1}u_n,u_n).
\end{equation}
Let $u^*=(u_1^*,\dots,u_n^*)$ be the point such that $z^*=\phi(u^*)$, so that $u_n^*\neq 0$.  
Then one easily sees
from \eqref{eq: defAssocBihomPoly}, \eqref{eq: coordu} and the homogeneity of $p$ that 
$
(P\circ\phi)(u,\conj{u})=p(|u_1|^2,\dots,|u_{n-1}|^2,1)\cdot |u_n|^{2d}$, so that
 \begin{equation}\label{eq: logPphi}
 \log (P\circ\phi)(u,\conj{u})=\log p(|u_1|^2,\dots,|u_{n-1}|^2,1)+d\cdot\log u_n+d\cdot\log\conj{u_n}
\end{equation}
near $u^*$ (for an appropriate logarithmic branch).  Hence one has
 \begin{align}\label{eq: d2logPphi}&\qquad
\displaystyle\frac{\partial^2( \log (P\circ\phi))} {\partial u_i\partial\conj{u_j}}(u,\conj{u})\\
&=\begin{cases}\displaystyle\Big(
u_j\conj{u_i}\cdot\frac{\partial^2( \log p)} {\partial x_i\partial x_j}
+\delta_{ij}\cdot\frac{\partial ( \log p)}{\partial x_i}\Big)
(|u_1|^2,\cdots,|u_{n-1}|^2,1)\quad\text{if }1\leq i,j\leq n-1,\\
0\quad\text{if }i=n\text{ or }j=n.
\end{cases}\notag
\end{align}
Now we take a tangent vector $0\neq v=
v_1\dfrac{\partial}{\partial u_1}+\cdots +v_n\dfrac{\partial}{\partial u_n}\in T_{u^*}(\C^n)$ satisfying 
$\partial\log(P\circ\phi)(v)=0$, or equivalently, 
\begin{equation}\label{eq: partialpphi}
\sum_{i=1}^{n-1}v_i\cdot \conj{u_i^*}\cdot\dfrac{\partial(\log p)  }{\partial x_i}(|u_1^*|^2,\cdots,|u_{n-1}^*|^2,1)
+v_n\cdot \dfrac{d}{u_n^*}=0
\end{equation}
(cf. \eqref{eq: logPphi}).   Together with the condition that  
$v\neq 0$, it follows readily that 
\begin{equation}\label{eq: voverunot0}
(v_1,\dots,v_{n-1})\neq (0,\dots,0).
\end{equation}
Let $\ell$ be the number of non-zero $u_i^*$'s for $1\leq i\leq n-1$ (so that $0\leq \ell\leq n-1$).
By permuting the first $n-1$ coordinate functions, we will assume without loss of generality that
$u_i^*\neq 0$ for each $1\leq i\leq \ell$ and $u_{\ell+1}^*=\cdots=u_{n-1}^*=0$. 
By using \eqref{eq: d2logPphi} and \eqref{eq: defineJ} (with $f=p_{\ell,\text{Id}}$ where 
$\text{Id}$ denotes the identity
permutation, and 
$s=(|u_1^*|^2,\dots,|u_{\ell}^*|^2,0,\dots,0,1)$), one easily checks that 
\begin{align}\label{eq: twoHessian}
\sum_{1\leq i,j\leq n}&\conj{v_i}\cdot
\frac{\partial^2( \log (P\circ\phi))} {\partial u_i\partial\conj{u_j}}(u^*,\conj{u^*})\cdot v_j =A_1+A_2,\quad\text{where}\\
\notag
A_1:&=\sum_{1\leq i,j\leq \ell} \dfrac{\conj{ v_i}}{ {u_i^*} }
 \cdot J_{p_{\ell,\text{Id}}}(|u_1^*|^2,\dots,|u_{\ell}^*|^2)_{ij}\cdot \dfrac{ v_j }{ \conj{u_j^*} }
\quad\text{and}
  \\
  \notag
A_2:= & \sum_{\ell+1\leq i\leq n-1}|v_i|^2\cdot\dfrac{\partial (\log p)}{\partial x_i}
 (|u_1^*|^2,\dots,|u_{\ell}^*|^2,0,\cdots,0,1).
\end{align}
Here $A_1$ (resp. $A_2$) is taken to be zero if $\ell=0$ (resp. $\ell=n-1$).
Note that $ (|u_1^*|^2,\dots,|u_{\ell}^*|^2,0,\cdots,0,1)\in F_i(\R_+^n)$ for each $\ell+1\leq i\leq n-1$.  
Hence from (Pos2), 
we see that $A_2>0$ whenever $\ell<n-1$ and  $(v_{\ell+1},\dots, v_{n-1})\neq (0,\dots,0)$.  
From  (Pos3) (for the set $\C^n \setminus (U(1) \cdot \PosOrthant)$) and the 
homogeneity of $p$ (for the set $U(1) \cdot \PosOrthant$), one easily sees that 
\begin{equation}\label{eq:  pinequalitycn}
|p(z) |\leq p(|z_1|,\dots,|z_n|)\quad\text{for all }z=(z_1,\dots,z_n) \in 
\C^n.
\end{equation}
Together with Lemma \ref{lem: Zspan}, it follows that one can apply Lemma \ref{lem: positiveHessianInInterior}
to conclude that $A_1>0$ whenever $\ell>0$ and $(v_1,\dots, v_{\ell})\neq (0,\dots, 0)$.  Since $n\geq 2$, 
by using \eqref{eq: voverunot0}, one easily concludes that $A_1+A_2>0$ 
in each of the three cases when $\ell=0$, $1\leq \ell<n-1$ or $\ell=n-1$.  This finishes the proof of 
\eqref{eq: pconvex}.
\end{proof}

\begin{proposition} \label{prop: pPSGCS}
If $p$ satisfies (Pos1), (Pos2) and (Pos3), then
\par\noindent
(i) for all $x=(x_1,\dots,x_n), ~y=(y_1,\dots,y_n) \in \R_+^n$, we have
\begin{equation}\label{eq: pxy2pxpy}
p(\sqrt{x_1y_1}, \ldots, \sqrt{x_ny_n})^2 \le p(x)\cdot p(y);\quad\text{and}
\end{equation}
(ii)
$P$ satisfies the SGCS inequality.
\end{proposition}

\begin{proof}
First we recall from Proposition \ref{prop: Pos1Pos3impliesPpd}
	that $p(x) > 0$ for all $x \in (\R_+^n)^{\circ}$.
Write $f := p_{n - 1, \mathrm{Id}}$
	where $p_{n - 1,\mathrm{Id}}$ is as in the proof of Proposition \ref{prop: StronglyPseudoconvex}
	(cf. also \eqref{eq: definepmsigma}),
			so that $f(s_1, \ldots, s_{n - 1}) = p (s_1, \ldots, s_{n - 1}, 1)$.
As in \eqref{eq: tildef},
	we consider the associated function $\widetilde f:\R^{n - 1}\to \R$
		given by 
		\begin{equation}\label{eq: tildefp}
		\tilde{f}(t_1, \ldots, t_{n - 1}) := f(e^{t_1}, \ldots, e^{t_{n - 1}})=p (e^{t_1}, \ldots, e^{t_{n - 1}}, 1).
		\end{equation}		
By Lemma \ref{lem: Zspan}(i),
	$N(f)$
		has affine dimension $n - 1$.
It also follows from \eqref{eq:  pinequalitycn}
	that $|f(z_1, \ldots, z_{n - 1})| \le f(|z_1|, \ldots, |z_{n - 1}|)$
		for all $(z_1, \ldots, z_{n - 1}) \in \C^{n - 1}$.			
Hence, by Lemma \ref{lem: positiveHessianInInterior} and \eqref{eq: HessJlogf},
	the Hessian matrix 
		$\displaystyle\Big( \frac{\partial^2}{\partial t_i \partial t_j} \log \tilde{f}(t)\Big)_{1 \le i, j \le n - 1}$			
		is positive definite for all $t \in \R^{n - 1}$,
			and it follows that $\log \tilde{f}$ is a convex function on $\R^{n - 1}$
				(see e.g. \cite[p. 37]{BL06}).   
In particular,
	we have 
$
\tilde{f}( \frac{t + t^\prime}{2}) \le \frac{1}{2}(\tilde{f}(t) + \tilde{f}(t^\prime))
$  for all $t=(t_1,\dots,t_{n-1}),~ t^\prime=(t_1^\prime,\dots,t_{n-1}^\prime) \in \R^{n - 1}$.
By letting $t_i=\log s_i,~t_i^\prime=\log s_i^\prime$ for each $i$, 
	it follows that we have
\begin{equation} \label{eq: convexEqn}
\log p(\sqrt{s_1 s_1^\prime}, \ldots, \sqrt{s_{n - 1} s_{n - 1}^\prime} , 1)
	\le \frac{1}{2} (\log p(s, 1) + \log p(s^\prime, 1))
\end{equation}
for all $s = (s_1, \ldots, s_{n - 1}) , s^\prime= (s_1^\prime, \ldots, s^\prime_{n - 1})\in (\R_+^{n - 1})^\circ$.
For any given $x , y \in (\R_+^n)^\circ$,
	by setting $s = (x_1/x_n, \ldots, x_{n - 1}/ x_n)$ and 
		$s^\prime = (y_1/y_n, \ldots, y_{n - 1}/y_n)$
			in \eqref{eq: convexEqn}, 
				and using the homogeneity of $p$,
		one easily sees that the inequality in \eqref{eq: pxy2pxpy}
			holds for such $x,y\in (\R_+^n)^\circ$.
Together with the continuity of $p$,	
	it follows that the inequality in \eqref{eq: pxy2pxpy} actually holds for all $x, y \in \R_+^n$,
	 	and this finishes the proof of (i).
For (ii), we let $z=(z_1,\dots,z_n)$, $w=(w_1,\dots,w_n)\in\C^n$ be linearly independent, which
implies readily that $(z_1 \conj{w_1}, \ldots, z_n \conj{w_n})\in \C^n \setminus (U(1) \cdot \PosOrthant)$.
Hence it follows from (Pos3) and (i) that
\begin{align} \label{eq: SGCSTemp1}
 |p(z_1 \conj{w_1}, \ldots, z_n \conj{w_n})|^2
                    &< p(|z_1||w_2|, \ldots, |z_n||w_n|)^2\\
                    &\leq p(|z_1|^2,\ldots, |z_n|^2) \cdot  p(|w_1|^2,\ldots, |w_n|^2),
                    \notag
\end{align}
which, together with \eqref{eq: defAssocBihomPoly}, imply that $|P(z, \conj{w})|^2
< P(z, \conj{z}) \cdot P(w, \conj{w})$, and this finishes the proof of (ii).
\end{proof}

We conclude this section with the following

\begin{proof}[Proof of Proposition \ref{prop: P3properties}]
Proposition \ref{prop: P3properties} follows directly from
Proposition \ref{prop: Pos1Pos3impliesPpd},
Proposition \ref{prop: StronglyPseudoconvex} and
Proposition \ref{prop: pPSGCS}.
\end{proof}

\section{Proof of Theorem \ref{thm: mainThm}} \label{sec: proofOfTheorem}

\begin{lemma} \label{lem: posCoeffs}
Let $f \in \R[x_1, \ldots, x_n]$ be a nonconstant homogeneous polynomial which has all positive coefficients.
Then 
\begin{equation} \label{eq: positiveOnOrthant}
f(x) > 0\quad \text{for all } x \in \R_+^n \setminus \{0\},
\end{equation}	and $f$ satisfies (Pos1), (Pos2) and (Pos3).
\end{lemma}

\begin{proof}
Write $f = \sum_{|I| = d} b_I x^I $, so that $d\geq 1$ and $b_I > 0$ for all $|I| = d$.
Then one easily sees that \eqref{eq: positiveOnOrthant} holds,
	which, in turn, implies that $f$ satisfies (Pos1).  For (Pos2), 
	we first consider the facet $F_1(\R_+^n)$.
	Let $x = (x_1, \ldots, x_n)\in F_1(\R_+^n) \setminus\{0\}$ be given,
	so that $x_1 = 0$ and $x_j > 0$ for some $1 < j \le n$.
Assume without loss of generality that $j = 2$.
Then 
\begin{equation} \label{eq: mthPowerHasPosDerivative}
\frac{\partial f}{\partial x_1}(x) 
    = \sum_{|I| = d} b_I I_1 x_1^{I_1 - 1} x_2^{I_2} \cdots x_n^{I_n}
	\ge b_{(1, d - 1, 0, \ldots, 0)} \cdot 1 \cdot x_2^{ d - 1} > 0.  
	\end{equation}
	The same argument yields 
the desired inequality on the other $F_k(\R_+^n)$'s, and this
finishes the proof of (Pos2).
For (Pos3), we
let $z = (z_1, \ldots, z_n) \in \C^n \setminus (U(1) \cdot \R_+^n)$ be given,
	so that $z_k \neq 0$ for some $1 \le k \le n$.
Since $b_I > 0$ for all $|I| = d$, we have
\begin{align}
|f(z)| = \big|\sum_{|I| = d} b_I z^I \big |
            \le \sum_{|I| = d} b_I |z_1|^{I_1} \cdots |z_n|^{I_n}
            = f(|z_1|,\ldots, |z_n|). \label{eq: cauchySchwarzPos3}
\end{align}	
If the inequality in \eqref{eq: cauchySchwarzPos3} is in fact an equality, 
then it is easy to see that all the $b_Iz^I$'s (and thus all the $z^I$'s) will have the same argument.
By comparing the arguments of $z_k^{d} $ and $z_k^{d - 1} z_j$ for each $1 \le  j \le n$,
	one sees that all the $z_j$'s have the same argument, contradicting the assumption that 
		$z \in \C^n \setminus (U(1) \cdot \R_+^n)$.  Hence the inequality in 
		\eqref{eq: cauchySchwarzPos3} is strict.
Thus $f$ satisfies (Pos3).		
\end{proof}

We are ready to give the proof of Theorem \ref{thm: mainThm} as follows:

\begin{proof}[Proof of Theorem \ref{thm: mainThm}]  Let $p \in \R[x_1,\dots,x_n]$ be a nonconstant homogeneous polynomial.
In the case when $n=1$, it is easy to see that $p$ satisfies \eqref{thm: Condition} (resp. 
\eqref{thm: EventualPos}) (in Theorem \ref{thm: mainThm}) if and only if $p$ is a monomial with positive coefficient. 
Hence we
only need to consider the case when $n\geq 2$.  
\par\noindent
$\underline{\eqref{thm: Condition}\implies\eqref{thm: EventualPos}}$:  
Suppose $p$ satisfies (Pos1), (Pos2) and (Pos3), and $q \in \R[x_1,\dots,x_n]$ is a homogeneous polynomial
such that $q(x)>0$ for all $x\in \R_+^n\setminus\{0\}$.  Let $P$ and $Q$ be 
the Hermitian bihomogeneous polynomial
associated to $p$ and $q$ respectively.  Then from Proposition \ref{prop: P3properties} and 
Theorem \ref{thm: CDThm}, one knows that  there exists $m_o > 0$ such that for each integer $m \ge m_o$,
            $P^m \cdot Q$ is a maximal squared norm.  By Proposition  \ref{prop: AssocBihomIsSOS},
            it follows that $p^mq$ has all positive coefficients for each such $m$.
            \par\noindent
$\underline{\eqref{thm: EventualPos} \implies \eqref{thm: Condition}}$:
By setting $q = 1$ in \eqref{thm: EventualPos}, 
	one knows that $p^m$ has all positive coefficients for some odd integer $m$. 
From Lemma \ref{lem: posCoeffs} (with $f = p^m$),
	one knows that $p(x)^m>0$ for all $ x \in \R_+^n \setminus \{0\}$, and $p^m$ satisfies (Pos1), (Pos2) and (Pos3).
By taking the $m$-th root, one immediately sees that $p(x)>0$ for all $ x \in \R_+^n \setminus \{0\}$, and 
$p$ also satisfies (Pos1) and (Pos3).
Since $p(x)>0$ for all $ x \in \R_+^n \setminus \{0\}$ and we have
\begin{equation}
\frac{\partial (p^m)}{\partial x_k}(x) = m \cdot p(x)^{m - 1} \cdot \frac{\partial p}{\partial x_k}(x) \quad\text{for each }
1\leq k\leq n,
\end{equation}
it follows readily from  (Pos2) for  $p^m$ that
$p$ also satisfies (Pos2).		
\end{proof}

\section{Application to polynomial spectral radius functions} \label{sec: Application}

In this section, we apply Theorem \ref{thm: mainThm} to 
prove Corollary \ref{cor: betaApplication},
and interpret Corollary \ref{cor: betaApplication} in terms of Markov chains.

Let $B = (B_{ij})$ be a square matrix with $B_{ij}\in\R_+$ for all $i,j$.  We recall that $B$ is said to be
{\emph{irreducible}}  if, for each pair of indices $i$ and $j$, there exists an  integer $k \ge 1$ such that 
$(B^k)_{ij} > 0$.  $B$ is said to be {\emph{aperiodic} if for each $i$, we have 
$\gcd\{ k \in \Z_+ \, \big| \, (B^k)_{ii} > 0 \} = 1$.
By the Perron-Frobenius theorem (see e.g. \cite{Se81}), if the matrix $B$ is irreducible or aperiodic, then its 
spectral radius $\beta(B)$ is positive.

Next we recall that an {\it irreducible}
(resp. {\it aperiodic}) {\it square matrix  $A = (A_{ij})$ over $\Z_+[x_1,\ldots, x_n]$} means that 
$A_{ij}\in \Z_+[x_1,\ldots, x_n]$ for all $i,j$, and for some (and hence all) $x\in (\R_+^n)^\circ$, the 
corresponding matrix $A(x)$ (with entries in $\R_+$) is irreducible (resp. aperiodic).  
In particular, for such $A$, we obtain its spectral radius function
$\beta_A :(\R_+^n)^\circ \to (0, \infty)$ given by $\beta_A(x) := \beta(A(x))$ for $x\in (\R_+^n)^\circ$.

In the probability theory of stochastic processes, 
an irreducible (resp.\ aperiodic) square matrix $A$ over $\Z[x_1, \ldots, x_n]$
defines an irreducible (resp.\ aperiodic) Markov chain $\Sigma_A$.
The spectral radius function $\beta_A$ is also known as the {\it beta function} of $\Sigma_A$ in Tuncel's paper \cite{Tu}
			(see also \cite{Deangelis942}).  
	The beta function is an important topological invariant in Markov shifts 
	(see e.g. \cite{MT93, PS84} and the references therein).
	In such context, Corollary \ref{cor: betaApplication} may be interpreted as a characterization 
	of certain polynomials as the beta functions of some irreducible or aperiodic Markov chains.

First we recall some results of De Angelis:
\begin{lemma}[{\cite[Theorem 3.3(i) (resp. Theorem 3.5)]{Deangelis942}}] 
    \label{lem: DeangelisOnRootOfBetaFns}
Let $p \in \Z[x_1, \ldots, x_n]$.
If there exists $m > 0$ 
    and an irreducible (resp. aperiodic) square matrix $B$ over $\Z_+[x_1, \ldots, x_n]$ 
        such that $p^m = \beta_B$, 
then $p = \beta_A$ 
    for some irreducible (resp. aperiodic) square matrix $A$ over $\Z_+[x_1, \ldots, x_n]$.
\end{lemma}

\begin{lemma}[{\cite[Theorem 6.6]{Deangelis942}}] \label{lem: deAngelisStarCondition}
Let $q \in \R[s_1, \ldots, s_\ell]$.
Suppose that the set 
$
S_q : =\{I - J \, \big| \, I, J \in \Log (q)\}
$ generates $\Z^\ell$ as a $\Z$-module, and $q = \beta_A$ for some square matrix $A$ over 
$\Z_+[s_1, \ldots, s_\ell]$.
   Then
$
|q(z)| < q(|z_1|, \ldots, |z_\ell|) $ for all $z=(z_1,\dots,z_\ell)\in \C^\ell\setminus \R_+^\ell$.
\end{lemma}

Finally we give the deduction of Corollary \ref{cor: betaApplication} as follows:

\begin{proof}[Proof of Corollary \ref{cor: betaApplication}]
Let $p \in \Z[x_1, \ldots, x_n]$ be a homogeneous polynomial which satisfies (Pos1) and (Pos2).
\par\noindent
$\underline{\eqref{cor: irr}  \implies \eqref{cor: pos3} ~(\text{resp. }
\eqref{cor: aper} \implies \eqref{cor: pos3} } $):
Suppose that there exists an irreducible
(resp. aperiodic) square matrix $A$ over $\Z_+[x_1, \ldots, x_n]$ such that $p = \beta_A$.
Let $p_{n - 1, \mathrm{Id}}$ be as in \eqref{eq: definepmsigma},
	so that
	\begin{equation} \label{eq: formulaForP}
	p_{n - 1, \mathrm{Id}}( s_1, \ldots, s_{n - 1}) = p(s_1, \ldots, s_{n -1}, 1).
	\end{equation}
Then $p_{n - 1, \mathrm{Id}} \in \Z [s_1, \ldots, s_{n - 1}]$
	and $p_{n - 1,\mathrm{Id}} = \beta_B$,
		where $B$ is the matrix over $\Z_+[s_1,\ldots, s_{n - 1}]$
			given by $B(s_1, \ldots, s_{n- 1}) = A(s_1, \ldots, s_{n - 1}, 1)$.
Furthermore, it follows from Lemma \ref{lem: Zspan}(ii) 
	that $S_{p_{n - 1, \mathrm{Id}}}$ generates $\Z^{ n - 1}$ as a $\Z$-module.
Thus by Lemma \ref{lem: deAngelisStarCondition}
	and \eqref{eq: formulaForP},
		we have
\begin{equation} \label{eq: hash}
|p(z_1, \ldots, z_{n - 1}, 1)| < p(|z_1|, \ldots, |z_{n - 1}|, 1)\quad
	\text{for all } (z_1, \ldots, z_{n - 1}) \in \C^{n - 1} \setminus  \R_+^{n-1}.
\end{equation}					
Next, we let $z = (z_1, \ldots, z_{n}) \in \C^n \setminus (U(1) \cdot \R_+^n)$, so that 
the $z_i$'s do not have the same argument.	
By permuting the coordinate functions, we may assume without loss of generality
	that $z_n \neq 0$, so that $(z_1/z_n, \ldots, z_{n - 1}/ z_n) \in \C^{n - 1} \setminus  \R_+^{n-1}$.		
Then it follows from \eqref{eq: hash}
	(with $z_i$ there replaced by $z_i/z_n$)
	that
	\begin{align}
	 \big| p\big(\frac{z_1}{z_n}, \ldots, \frac{z_{n - 1}}{z_n} , 1\big)\big| 
	 	< p\big(\big|\frac{z_1}{z_n}\big| ,\ldots, \big|\frac{z_{n - 1}}{z_n}\big|, 1\big) 
	\implies  |p(z)|  < p(|z_1|, \ldots, |z_n|), 
	\end{align}
	where the implication follows from the homogenity of $p$.
	Hence $p$ satisfies (Pos3).
\par\noindent
$\underline{\eqref{cor: pos3} \implies \eqref{cor: irr}~(\text{resp. }\eqref{cor: pos3} 
\implies \eqref{cor: aper})}$: Suppose that $p$ also satisfies (Pos3).
Then by Theorem \ref{thm: mainThm} (with $q = 1$ in \eqref{thm: EventualPos}),
	there exists $m > 0$ such that $p^m$ has all positive coefficients.
Since $p^m$ is nonzero, 
    the $1 \times 1$ matrix $B := (p^m)$ is irreducible (resp. aperiodic)  over $\Z_+[x_1,\dots,x_n]$, 
    and $\beta_B = p^m$.
Hence by Lemma \ref{lem: DeangelisOnRootOfBetaFns},
    there exists 
        an irreducible (resp. aperiodic) 
            square matrix $A$ over $\Z_+[x_1, \ldots, x_n]$ 
                such that $p = \beta_A$.
\end{proof}


\begin{thebibliography}{99}

\bibitem{BE15} W. Bergweiler and A. Eremenko,
Distribution of zeros of polynomials with positive coefficients.
 Ann. Acad. Sci. Fenn. 
\textbf{40} (2015) 375-383. 

\bibitem{BL06}
J. Borwein and A. Lewis, {\it Convex analysis and nonlinear optimization. 
Theory and examples.} Second edition. 
Springer, New York, 2006.

\bibitem{CD97} D. Catlin and J. D'Angelo, Positivity conditions for bihomogeneous polynomials. Math. Res. Lett.  {\bf 4} (1997),  555-567.

\bibitem{CD99} D. Catlin and J. D'Angelo, An isometric imbedding theorem for holomorphic bundles. Math. Res. Lett.  {\bf 6} (1999),  43-60.

\bibitem{DAngelo02} 
J. D'Angelo, {\it Inequalities from complex analysis}. Carus Mathematical Monographs, No. 28. Mathematical Association of America, Washington, DC, 2002. 

\bibitem{DV04} J. D'Angelo and D. Varolin, Postivity conditions for Hermitian symmetric functions. Asian J. Math. {\bf 8} (2004), 215-232.



\bibitem{Deangelis94}
V. De Angelis, Positivity conditions for polynomials. 
Ergodic Theory Dynam. Systems {\textbf{14}} (1994), no. 1, 23-51.

\bibitem{Deangelis942}
V. De Angelis, Polynomial beta functions.
Ergodic Theory Dynam. Systems {\textbf{14}} (1994), no. 3, 453-474.


\bibitem{Deangelis03}
V. De Angelis, Asymptotic expansions and positivity of coefficients for
large powers of analytic functions.
 Int. J. Math. Math. Sci. \textbf{16} (2003), 1003-1025.

\bibitem{DP02}
E. de Klerk and D.-V. Pasechnik Approximation of the stability number of a graph via copositive programming. 
SIAM J. Optim. {\textbf{12}} (2002), no. 4, 875-892.

\bibitem{Ere14}
A. Eremenko, Stability of real polynomials with positive coefficients, URL (version: 2014-09-16): http://mathoverflow.net/q/180493

\bibitem{EF} A. Eremenko and A. Fryntov,
Remarks on Obrechkoff's inequality.
Proc. Amer. Math. Soc. \textbf{144}, 2 (2016), 703-707. 

\bibitem{FG2002}
K. Fritzsche and H. Grauert, {\it From holomorphic functions to complex manifolds}. 
Springer-Verlag, New York, 2002.


\bibitem{HL13}
J. Halfpap and J. Lebl, Signature pairs of positive polynomials. Bull. Inst. Math. Acad. Sin. (N.S.)  {\bf 8}  (2013),  no. 2, 169-192.


\bibitem{Han86}
D. Handelman, Deciding eventual positivity of polynomials.
Ergodic Theory Dynam. Systems \textbf{6} (1986), 342-350.






\bibitem{MT93}
B. Marcus and S. Tuncel, Matrices of polynomials, positivity, and finite equivalence of Markov chains. 
J. Amer. Math. Soc.  {\bf 6}  (1993),  no. 1, 131-147. 

 
 \bibitem{PS84}
 W. Parry and K. Schmidt, Natural coefficients and invariants for Markov-shifts. Invent. Math.  
 {\bf 76}  (1984),  no. 1, 15-32.

\bibitem{Polya28}
G. P\'{o}lya, \"{U}ber positive Darstellung von Polynomen. Vierteljschr.
Naturforsch. Ges. Z \"{u}rich
\textbf{73} (1928), 141-145, in Collected Papers \textbf{2} (1974), MIT Press, 309-313.

\bibitem{PR09}
V. Powers and B. Reznick, A new bound for P\'{o}lya's Theorem with applications to polynomials 
positive on polyhedra. 
J. Pure Appl. Algebra {\textbf{164}} (2001), no. 1-2, 221-229.


\bibitem{Reznick95} B. Reznick, Uniform denominators in Hilbert's seventeenth problem. Math. Z. {\bf 220} (1995), 75-97.


\bibitem{Schmudgen91}
K. Schm\"{u}dgen, 
The K-moment problem for compact semi-algebraic sets. Math. Ann. {\bf 289} (1991), no. 2, 203-206.

\bibitem{Schweighofer04}
M. Schweighofer, 
On the complexity of Schm\"{u}dgen's Positivstellensatz. J. Complexity \textbf{20} (2004), no. 4, 529-543.

\bibitem{Se81}
E. Seneta, {\it Nonnegative matrices and Markov chains}. Second edition. 
Springer Series in Statistics. Springer-Verlag, New York, 1981.


\bibitem{Tu} 
S. Tuncel, 
Conditional pressure and coding. Israel J. Math.  {\bf 39}  (1981), no. 1-2, 101-112. 

\bibitem{TY06} W.-K. To and S.-K. Yeung, Effective isometric embeddings for certain Hermitian holomorphic line bundles. J. London Math. Soc.  {\bf 73}  (2006),  607-624.

\bibitem{Var08} D. Varolin, Geometry of Hermitian algebraic functions. Quotients of squared norms. Amer. J. Math.  {\textbf{130}}  (2008), 291-315.

\end{thebibliography}
\end{document}